\documentclass[12pt]{amsart}
\usepackage{amscd,amsmath,amsthm,amssymb}
\usepackage{pstcol,pst-plot,pst-3d}
\usepackage{lmodern,pst-node}
\def\NZQ{\Bbb}               
\def\NN{{\NZQ N}}

\def\ZZ{{\NZQ Z}}
\def\RR{{\NZQ R}}

\def\frk{\frak}               

\def\mm{{\frk m}}

\def\Phi{{\frk n}}
\def\Phi{{\frk N}}
\def\MP{{\mathcal P}}

\def\MV{{\mathcal V}}

\def\MF{{\mathcal F}}
\def\MT{{\mathcal T}}
\def\MH{{\mathcal H}}
\def\MS{{\mathcal S}}

\def\opn#1#2{\def#1{\operatorname{#2}}} 
\opn\chara{char} \opn\length{\ell} \opn\pd{pd} \opn\rk{rk}
\opn\projdim{proj\,dim} \opn\injdim{inj\,dim} \opn\rank{rank}
\opn\depth{depth} \opn\grade{grade} \opn\height{height}
\opn\embdim{emb\,dim} \opn\codim{codim}

\opn\Tr{Tr} \opn\bigrank{big\,rank}
\opn\superheight{superheight}\opn\lcm{lcm}
\opn\trdeg{tr\,deg}
\opn\reg{reg} \opn\lreg{lreg} \opn\ini{in} \opn\lpd{lpd}
\opn\size{size}\opn\bigsize{bigsize}
\opn\cosize{cosize}\opn\bigcosize{bigcosize}
\opn\sdepth{sdepth}\opn\sreg{sreg}
\opn\link{link}\opn\fdepth{fdepth}
\opn\div{div} \opn\Div{Div} \opn\cl{cl} \opn\Cl{Cl}
\opn\Spec{Spec} \opn\Supp{Supp} \opn\supp{supp} \opn\Sing{Sing}
\opn\Ass{Ass} \opn\Min{Min}\opn\Mon{Mon} \opn\dstab{dstab} \opn\astab{astab}
\opn\Ann{Ann} \opn\Rad{Rad} \opn\Soc{Soc}
\opn\Im{Im} \opn\Ker{Ker} \opn\Coker{Coker} \opn\Am{Am}
\opn\Hom{Hom} \opn\Tor{Tor} \opn\Ext{Ext} \opn\End{End}
\opn\Aut{Aut} \opn\id{id}

\opn\nat{nat}
\opn\pff{pf}
\opn\Pf{Pf} \opn\GL{GL} \opn\SL{SL} \opn\mod{mod} \opn\ord{ord}
\opn\Gin{Gin} \opn\Hilb{Hilb}\opn\sort{sort}
\opn\aff{aff} \opn\con{conv} \opn\relint{relint} \opn\st{st}
\opn\lk{lk} \opn\cn{cn} \opn\core{core} \opn\vol{vol}
\opn\link{link} \opn\star{star}\opn\lex{lex}
\opn\gr{gr}
\def\Rees{{\mathcal R}}
\def\pot#1#2{#1[\kern-0.28ex[#2]\kern-0.28ex]}

\opn\dirlim{\underrightarrow{\lim}}
\opn\inivlim{\underleftarrow{\lim}}
\let\union=\cup
\let\sect=\cap
\let\dirsum=\oplus

\let\iso=\cong
\let\Union=\bigcup
\let\Sect=\bigcap

\let\to=\rightarrow
\let\To=\longrightarrow
\def\Implies{\ifmmode\Longrightarrow \else
        \unskip${}\Longrightarrow{}$\ignorespaces\fi}
\def\implies{\ifmmode\Rightarrow \else
        \unskip${}\Rightarrow{}$\ignorespaces\fi}
\def\iff{\ifmmode\Longleftrightarrow \else
        \unskip${}\Longleftrightarrow{}$\ignorespaces\fi}

\let\:=\colon
\newtheorem{Theorem}{Theorem}[section]
\newtheorem{Lemma}[Theorem]{Lemma}
\newtheorem{Corollary}[Theorem]{Corollary}
\newtheorem{Proposition}[Theorem]{Proposition}
\newtheorem{Remark}[Theorem]{Remark}

\newtheorem{Example}[Theorem]{Example}

\newtheorem{Algorithm}[Theorem]{Algorithm}

\let\epsilon\varepsilon
\let\kappa=\varkappa
\def\qed{\ifhmode\textqed\fi
      \ifmmode\ifinner\quad\qedsymbol\else\dispqed\fi\fi}
\def\textqed{\unskip\nobreak\penalty50
       \hskip2em\hbox{}\nobreak\hfil\qedsymbol
       \parfillskip=0pt \finalhyphendemerits=0}
\def\dispqed{\rlap{\qquad\qedsymbol}}

\opn\dis{dis}
\def\pnt{{\raise0.5mm\hbox{\large\bf.}}}

\opn\Lex{Lex}

\begin{document}

\title {The stable set of associated prime ideals of a polymatroidal ideal }
\author {J\"urgen Herzog, Asia Rauf and Marius Vladoiu}
\address{J\"urgen Herzog, Fachbereich Mathematik, Universit\"at Duisburg-Essen, Campus Essen, 45117
Essen, Germany} \email{juergen.herzog@uni-essen.de}

\address{Asia Rauf, The Abdus Salam International Centre for Theoretical Physics (ICTP), Trieste-Italy.} \email{arauf@ictp.it, asia.rauf@gmail.com}

\address{Marius Vladoiu, Faculty of Mathematics and Computer Science, University of Bucharest, Str.
 Academiei 14, Bucharest, RO-010014, Romania}\email{vladoiu@gta.math.unibuc.ro}

\subjclass{13C13, 13A30, 13F99,  05E40}
\keywords{Associated prime ideals, polymatroidal ideals, analytic spread}
\begin{abstract}
The associated prime ideals of powers of polymatroidal ideals are studied, including the stable set of associated prime ideals of this class of ideals. It is shown that polymatroidal ideals have the persistence property and  for transversal polymatroids and polymatroidal ideals of Veronese type  the index of stability and the stable set of associated ideals is determined explicitly.
\end{abstract}
\maketitle

\section*{Introduction}
Let $I$ be an ideal in a Noetherian ring $R$. It is customary to denote by $\Ass(I)$ the set of associated prime ideals of $R/I$. Brodmann \cite{Br2} showed that $\Ass(I^k)=\Ass(I^{k+1})$ for all $k\gg 0$. One calls the smallest number $k_0$ for which this happens the {\em index of  stability} and  $\Ass(I^{k_0})$ is called  the {\em stable set of associated prime ideals} of $I$. It is denoted by $\Ass^\infty(I)$. Several natural questions arise in the context of Brodmann's theorem.

(1) Is there an upper bound for the index of stability depending only on $R$?

(2) What can be said about the set $\Ass^\infty(I)$? Can $\Ass^\infty(I)$ be computed in case that $R$ is a polynomial ring and $I$ is a graded  ideal?

(3) Is it true that $\Ass(I)\subset \Ass(I^2)\subset \ldots \subset \Ass(I^k)\subset\ldots$?

All these questions are widely open, even for monomial ideals, though in several interesting special cases, including edge ideals and vertex cover ideals, these questions have been answered quite comprehensively, see \cite{CMS}, \cite{FHV} and \cite{MMV}. A nice survey on what is known about the stability of associated prime ideals of powers of edge ideals is given in \cite{MV}. Question (3) doesn't have a positive answer in general, see \cite{HH1} and \cite{MMV} for counterexamples. The ideals which provide these counterexamples are monomial ideals, but not squarefree. An ideal $I$ for which (3) holds true is said to satisfy the {\em persistence  property}. It is an open question whether all squarefree monomial ideals satisfy the persistence property.

Suppose now that $(R,\mm)$ is local or a standard graded $K$-algebra with graded maximal ideal $\mm$. We say that an ideal $I\subset R$ has  {\em  non-increasing depth functions}, if for all  prime ideals $P$ in the support $V(I)$ of $R/I$, one has that  $\depth R_P/I^kR_P$ is a non-increasing function of $k$, and  $I$ is said to have  {\em  strictly decreasing depth functions}, if the depth functions of all its localizations are strictly decreasing until they reach their limit value.  In the case that $I$ is a graded  ideal, respectively a monomial ideal, we require the defining property of non-increasing (strictly decreasing)  depth functions  only for localizations with respect to prime ideals $P\in V^*(I)$, where $V^*(I)$ denotes  the set of   graded, respectively  monomial prime ideals containing $I$.

It is easily seen  that for an ideal which has non-increasing depth functions   the persistence property holds, see Proposition~\ref{basis}.  Moreover, if an ideal has strictly decreasing depth functions, then its index of stability is bounded by $\dim R-1$.  We do not know of any example of a squarefree monomial ideal which does not have non-increasing  depth functions. On the other hand it is shown in \cite[Theorem 4.1]{HH1} that, given any non-decreasing function $f\: \NN\to \NN$,  there exists a monomial ideal in a polynomial ring $S$ (with sufficiently many variables) such that $\depth S/I^k=f(k)$ for all $k$. This shows that among the monomial ideals, non-increasing depth functions  can be expected in general only for squarefree monomial ideals.

There is at least one case known to us in which  $\Ass^\infty(I)$ can be computed efficiently. Namely, if $I$ is a monomial ideal in a polynomial ring $S=K[x_1,\ldots,x_n]$ whose Rees algebra $\Rees(I)$ is Cohen-Macaulay. By a result of Huneke \cite{Hu} it follows that the associated graded ring of $I$ is Cohen--Macaulay, and this implies that $\lim_k\depth  S/I^k= n-\ell(I)$, where $\ell(I)$ denotes the analytic spread of $I$, that is, the Krull dimension of the fiber ring $\Rees(I)/\mm \Rees(I)$, see Eisenbud and Huneke \cite[Proposition 3.3]{EH}. This theorem allows us to identify the elements of $\Ass^\infty(I)$ in terms of the exponent matrix associated with the unique minimal monomial set $G(I)$ of generators of $I$. All this is explained in detail in Section 1. There we also define the invariants $\dstab(I)$ and $\astab(I)$. The first of them is the smallest integer $k$ with the property that $\depth I^k=\depth I^{\ell}$ for all $\ell\geq k$, while the second is the smallest integer with  $\Ass(I^k)=\Ass(I^{\ell})$ for all $\ell \geq k$. One may ask whether there is any relation between these numbers. At the end of Section 1 we show that either one may be smaller than the other.  However we show in Proposition~\ref{basis} that $\astab(I)$ is bounded below and above by local data of $\dstab$ provided it has non-increasing depth functions, which is for example the case if all powers of $I$ have a linear resolution, see Proposition~\ref{linear}. These facts are used to compute the index of stability for Stanley-Reisner  ideal  of the natural triangulation of the projective plane.

In Section 2  the strategies discussed in  Section 1 are applied to study the associated prime ideals of powers of polymatroidal ideals. Two   general properties of polymatroidal ideals are crucial: 1.\ all powers of polymatroidal ideals have a linear resolution, as shown in \cite[Theorem 5.3]{CH}, 2.\ localizations of polymatroidal ideals at monomial prime ideals are again polymatroidal, see Corollary~\ref{romanianstyle}. These two facts combined with Proposition~\ref{basis} immediately yield that polymatroidal ideals have the persistence property, see Proposition~\ref{persistence}. We recall in Theorem~\ref{normality} the result of Villarreal \cite[Proposition 3.11]{Vi}, which says that the Rees ring of a polymatroidal ideal is normal, and consequently Cohen--Macaulay. Applying then the Huneke-Eisenbud result, the limit depth of a polymatroidal ideal can be expressed by its analytic spread. From this one easily deduces an algorithm, described at the end of the section,  to compute $\Ass^\infty(I)$ for any polymatroidal ideal $I$. All data required to compute  $\Ass^\infty(I)$ are given by the exponent matrix of the minimal set of monomial generators of $I$.

In the remaining two sections we consider special classes of polymatroidal ideals where the questions concerning associated prime ideals of powers of ideals have  complete answers. The ideals considered in Section 3 are the polymatroidal ideals of transversal polymatroids. Algebraically speaking, ideals of this type are simply arbitrary (finite) products of monomial prime ideals. In \cite[Lemma 3.2]{CH} a primary decomposition of products of ideals generated by linear forms is given. However this decomposition is not at all irredundant and it is not easy to obtain an irredundant decomposition from that given in \cite[Lemma 3.2]{CH}.

Our first result (Lemma~\ref{unique}) asserts that the presentation of a transversal polymatroidal ideal as product of monomial prime ideals is unique. The key result of Section~3 is Theorem~\ref{maxgraphconnected} where it is shown that  the graded maximal ideal $\mm$ is associated to  the transversal polymatroidal ideal $I=P_{F_1}\cdots P_{F_r}$ if and only if  $\Union_{i=1}^r F_i=[n]$ and the intersection graph $G_I$ is connected. Here $G_I$ is the graph with vertex set  $\{1,\ldots,r\}$ and for which  $\{i,j\}$ is an edge of $G_I$ if and only if $F_i\sect F_j\neq\emptyset$. By using this result we conclude in Corollary~\ref{astab1} that $\Ass(I)=\Ass^\infty(I)$ for any transversal polymatroidal ideal. Furthermore we show in Theorem~\ref{asstransversal} that $\Ass(I)$ is determined by the trees of the graph $G_I$. As nice consequences of all this we classify  in Corollary~\ref{assprimes} all subsets $\MS=\{F_1,\ldots,F_r\}$ of  $2^{[n]}$ for which there exists a transversal polymatroidal ideal $I$ with $\Ass(I)=\{P_{F_1},\ldots,P_{F_r}\}$, and in Corollary~\ref{primarydecomp} we give an irredundant primary decomposition of all powers $I^k$ of $I$. We conclude this section with two results concerning the depth of $S/I$. In Theorem~\ref{depthtransversal} it is shown that $\depth S/I$ is essentially determined by the number of components of $G_I$, and  Corollary~\ref{limitdepthtran}  says that $\dstab(I)=1$. Thus for any transversal polymatroidal ideal, $\dstab(I)=\astab(I)=1$.

The situation for ideals  $I_{d;a_1,\ldots,a_n}$ of Veronese type, which is the class of polymatroidal ideals considered in Section 4, is completely different. Here $\Ass^\infty(I)=V^*(I)$, as shown in Proposition~\ref{stableassveronese}, and the invariant $\astab(I)$ can be any number between 1 and $n-1$ determined by an explicit formula given in terms of the numbers $d$ and $a_1,\ldots,a_n$, see Corollary~\ref{primeindex}. Moreover it is shown in Corollary~\ref{dstabveronese} that $\astab(I)=\dstab(I)$ and $\lim_{k\to\infty}\depth S/I^k$ and $\ell(I)$ are computed for any Veronese type ideal.

The common feature to transversal polymatroidal ideals and to ideals of Veronese type is that $\astab(I)=\dstab(I)$. It would be interesting to know whether this equality holds for any other polymatroidal ideal. As we have seen in Section 1, arbitrary monomial ideals, even when they are squarefree, do not satisfy this equality.

\section{Generalities about the depth and the associated primes of powers of an ideal}

 Let $(R,\mm)$ denote a Noetherian local ring or standard graded $K$-algebra with graded maximal ideal $\mm$,
 and  $I\subset R$ an ideal. In the graded case we assume that $I$
 is graded ideal.

 We are going to  relate the index of stability  and the persistence property of $I$ to the property of $I$ to have non-increasing depth functions. We say that $P\in V(I)$  is a {\em persistent prime ideal} of $I$, if whenever $P\in \Ass(I^k)$ for some exponent $k$, then $P\in \Ass(I^{k+1})$. If this happens to be so for $k$, then of course we have $P\in \Ass(I^{\ell})$ for all $\ell\geq k$. The ideal $I$ is said to have the {\em persistence property} if all prime ideals $P\in \Union_k\Ass(I^k)$ are persistent prime ideals.

 By a famous theorem of Brodmann \cite{Br1} it is known that $\depth R/I^k$ is constant for all $k\gg 0$. We call the smallest  number $k_0$ such that $\depth R/I^k=\depth R/I^{k_0}$ for all $k\geq k_0$, the {\em index of depth stability} of $I$, and denote this number by $\dstab(I)$.

 Brodmann also showed \cite{Br2} that there exists an integer $k_1$ such that $\Ass(I^k)=\Ass(I^{k_1})$  for all $k\geq k_1$. The smallest such number is called the {\em index of stability} of $I$. We denote this number by $\astab(I)$.

At the end of this section we show by examples that the  invariants $\dstab(I)$ and $\astab(I)$ are unrelated.  In other words, either one of these numbers may be smaller than the other one or they may also be equal. However we have

\begin{Proposition}
\label{basis}
{\em (a)} Suppose the depth function $\depth R/I^k$ is non-increasing, then $\mm$ is a persistent prime ideal.

{\em (b)} If $I$ has non-increasing depth functions, then $I$ satisfies the persistence property.

{\em (c)} $\max_{P\in\Ass^{\infty}(I)}\{\dstab(IR_P)\}\leq\astab(I)$. In addition,  if $I$ has non-increasing depth functions, then $\astab(I)\leq \max_{ P\in V(I)}\{\dstab(IR_P)\}$.
\end{Proposition}

\begin{proof}
(a) Let $\mm\in \Ass(I^k)$, then $\depth R/I^k=0$. Thus our assumption implies that  $\depth R/I^\ell=0$ for all $\ell\geq k$. Hence $\mm\in Ass(I^\ell)$ for all $\ell\geq k$.

(b) One has that $P\in \Ass(I^k)$ if and only if $PR_P\in \Ass_{R_P}(I^kR_P)$. By part (a) this is implies that $PR_P\in \Ass_{R_P}(I^\ell R_P)$ for all $\ell\geq k$. Thus  $P\in \Ass(I^\ell)$ for all $\ell\geq k$.

(c) Let $r=\astab(I)$. Then,  whenever $P\in \Ass^\infty(I)$, we have  $P\in \Ass(I^\ell)$ for all $\ell\geq r$. This implies that $\depth R_P/I^\ell R_P=0$ for all $\ell\geq r$. Hence $\dstab(IR_P)\leq r$, which yields the first inequality.

Now let $s=\max_{P\in V(I)}\{\dstab(IR_P)\}$, and suppose that $r>s$. Then there exists $P\in \Ass^\infty(I)$ such that $P\in \Ass(I^r)$, but $P\not\in \Ass(I^s)$. Indeed, otherwise we would have that $\depth R_P/I^sR_P=0$ for all $P\in\Ass^\infty(I)$. Since $I$ has non-increasing depth functions it would follow that $\astab(I)\leq s<r$, a contradiction.

    It follows that $\depth R_P/I^sR_P>\depth R_P/I^rR_P=0$, in contradiction to the definition of $s$.
\end{proof}

The next result generalizes \cite[Proposition 2.1]{HH1} and provides cases where we have non-increasing depth functions.

\begin{Proposition}
\label{linear}
Suppose $I\subset S=K[x_1,\ldots,x_n]$ is a graded ideal generated in degree $d$ with the property that there exists an integer $k_0$ such that  $I^k$ has a linear resolution for all $k\geq k_0$. Then $\depth I^k\geq \depth I^{k+1}$ for all $k\geq k_0$.
\end{Proposition}

\begin{proof}
Let $f\in I$ be a homogeneous polynomial of degree $d$. Then $fI^k$ is generated in degree $(k+1)d$ and $fI^k\subset I^{k+1}$. The short exact sequence
\[
0\To fI^k \To I^{k+1}\To I^{k+1}/fI^k\To 0
\]
induces the long exact sequence
\[
\ldots \to\Tor_{i+1}(K,I^{k+1}/fI^k)_{i+1+(j-1)}\to \Tor_{i}(K,fI^k)_{i+j}\to \Tor_{i}(K,I^{k+1})_{i+j}\to\ldots ,
\]
where for a graded $S$-module, $\Tor_i(K,M)_j$ denotes the $j$th graded component of $\Tor_i(K,M)$.

Both $fI^k$ and $I^{k+1}$ have a $(k+1)d$-linear resolution. Thus  $$\Tor_{i}(K,fI^k)_{i+j}= \Tor_{i}(K,I^{k+1})_{i+j}=0$$ for $j\neq (k+1)d$ and all $i$. Moreover, $\Tor_{i+1}(K,I^{k+1}/fI^k)_{i+1+(j-1)}=0$ for $j=(k+1)d$, because the module $I^{k+1}/fI^k$ is generated in degree $(k+1)d$. This  shows that the natural maps
$\Tor_{i}(K,fI^k)\To \Tor_{i}(K,I^{k+1})$ are injective for all $i$. It follows that $\projdim I^k=\projdim fI^k\leq \projdim I^{k+1}$, and consequently,
$\depth S/I^{k+1}\leq \depth S/I^k$, by the Auslander-Buchsbaum formula (see for example \cite[Theorem 1.3.3]{BH}).
\end{proof}

Let $I\subset S$ be a monomial ideal. Throughout this paper $S$ stands for the polynomial ring $K[x_1,\ldots,x_n]$ where $K$ is a field. We denote by $G(I)$ the unique minimal set of monomial generators of $I$. In the case that $G(I)\subset T= K[x_{i_1},\ldots, x_{i_k}]$ we denote by an abuse of notation the ideal $G(I)T$ again by $I$. Observe that by using this notation it follows that $\Ass_S(I)=\Ass_T(I)$.

Let $u=\prod_{i\in L}x_i$ be a squarefree monomial in $S$. Then $$(S/I)_u\iso S'[\{x_j^{\pm 1}\:\; j\in L\}]/I_LS'[\{x_j^{\pm 1}\:\; j\in L\}],$$
where $S'=K[\{x_i\:\; i\notin L\}]$ and where $I_L\subset S'$  is the ideal which is obtained  from $I$ by applying  the $K$-algebra homomorphism $S\to S'$ with $x_i\mapsto 1$ for all $i\in L$.

\medskip
Let $P=(x_{i_1},\ldots,x_{i_r})$ be a monomial prime ideal, and $I\subset S$ any monomial ideal. We denote by $I(P)$ the monomial ideal in the polynomial ring $S(P) =K[x_{i_1},\ldots,x_{i_r}]$ where $I(P)=I_L$ with  $L=[n]\setminus \{i_1,\ldots, i_r\}$.

In the  later proofs  we need the following simple facts.

\begin{Lemma}
\label{localass}
Let $I\subset S$ be a monomial ideal. Then
\begin{enumerate}
\item[(a)]  $P\in \Ass(I)$ if and only if $\depth S(P)/I(P)=0$;
\item[(b)] $\Ass(I_L)=\{P\in \Ass(I)\:\; x_i\not\in P \text{ for all $i\in L$}\}$ for all subsets $L\subset [n]$.
\end{enumerate}
\end{Lemma}

\begin{proof}
(a) has been observed in \cite[Lemma 2.11]{FHV}.

(b) As before let $S'=K[\{x_i\:\; i\notin L\}]$ and set $T=S'[\{x_j^{\pm 1}\:\; j\in L\}]$. Then $T=S_u$ where $u=\prod_{i\in L}x_i$. Thus by using the basic  rules concerning the behaviour of associated prime ideals with respect to localization and polynomial ring extension we obtain
\[
\Ass_T(I_LT)=\Ass_T(IT)=\{PT\:\; P\in \Ass_S(I),  x_i\notin P  \text{ for all $i\in L$}\}.
\]
On the  other hand,
\[
\Ass_T(I_LT)=\{PT\:\; P\in \Ass_S(I_L)\}.
\]
Since the assignment $P\mapsto PT$ establishes a bijection between the  set $\Ass_{S'}(I_L)$ and $\{PT\:\; P\in \Ass_S(I_L)\}$, the desired conclusion follows.
\end{proof}

If $I$ is a   monomial ideal, we  say that  $I$ has  non-increasing depth functions if $\depth S(P)/I(P)^k$ is a non-increasing function of $k$ for all $P\in V^*(I)$.

Since the associated prime ideals of a monomial  ideal  are monomial prime ideals, it follows (in analogy to Proposition~\ref{basis}(b))  that a monomial ideal has the persistence property if   $I$ has non-increasing depth functions as  defined for monomial ideals.

For monomial ideals, the corresponding statement of Proposition~\ref{basis}(c)  reads as follows:

\begin{Proposition}
\label{newread}
Let $I\subset S$ be a monomial ideal which has non-increasing depth functions. Then  $$\max_{P\in\Ass^{\infty}(I)}\{\dstab(I(P))\}\leq \astab(I)\leq \max_{P\in V^*(I)}\{\dstab(I(P))\}.$$
In particular, if $\Ass^\infty(I)=V^*(I)$, one has $\astab(I)= \max_{P\in V^*(I)}\{\dstab(I(P))\}$.
\end{Proposition}

As a consequence, in the case of a monomial ideal which has non-increasing depth functions we need to compute the depth stability only for a finite number of monomial prime ideals in order to obtain bounds for its index of stability. The following example demonstrates this strategy.

\medskip
Let $I$ be the Stanley-Reisner ideal that corresponds to the natural triangulation of the projective plane. Then
\[
I = (x_1x_2x_3, x_1x_2x_4, x_1x_3x_5, x_1x_4x_6, x_1x_5x_6, x_2x_3x_6, x_2x_4x_5, x_2x_5x_6, x_3x_4x_5,
x_3x_4x_6).
\]
The following table displays in the $j$th row and the $k$th column the depth of $S(P)/I(P)^k$ where  $P\in V^*(I)$  is of height $j$.
\medskip
\begin{center}
\begin{tabular}{l |*{4}{c}r}
I           & 1 & 2 & 3 & 4 \\
\hline
3           & 0 & 0 & 0 & 0 \\
4           & 1 & 1 & 1 & 1 \\
5           & 2 & 2 & 0 & 0 \\
6           & 3 & 0 & 0 & 0 \\
\end{tabular}
\end{center}
The ideal $I$ is of height 3, so $\depth  S(P)/I(P)^k=\dim S(P)/I(P)^k=0$ for all $k$ and all $P\in V^*(I)$ of height $3$. For $P\in  V^*(I)$  of height $4$ the ideal $I(P)$ is a monomial ideal complete intersection of height $3$. Therefore $\depth S(P)/I(P)^k=1$ for all $k\geq 1$. This explains the first two rows of the table. If $P\in  V^*(I)$ is of height $5$, then $I(P)$ is the edge ideal of a $5$-cycle. It follows from \cite[Lemma 3.1]{CMS} that $\Ass^{\infty}(I(P))=\Ass(I(P))\union\{\mm\}$ and $\astab(I(P))=3$, which explains the third row of the table. In particular, we have that $\depth S(P)/I(P)^k=0$ for all $P\in V^*(I)$ of height $5$ and for all $k\geq 3$. Finally, by using CoCoA \cite{C} we find that $\depth S/I=3$, $\depth S/I^2=0$ and  $\depth S/I^3=0$. Borna \cite[Corollary 3.3]{Bo} has shown that $I^k$ has a linear resolution for $k\geq 3$ and when $\chara(K)=0$. Applying Proposition~\ref{linear}, we see that $\depth S/I^k=0$ for all $k\geq 3$. It follows also that $I$ is an ideal with non-increasing depth functions and consequently $I$ satisfies the persistence property, by Proposition~\ref{basis}(b). By applying the second inequality of Proposition~\ref{newread} we obtain that $\astab I\leq 3$. By using Singular \cite{S}, we find that all prime ideals of height $5$ are in $\Ass(I^3)$. Since $I$ satisfies persistence property we get that all prime ideals of height $5$ are in $\Ass^{\infty}(I)$. It follows then that $\max_{P\in\Ass^{\infty}(I)}\{\dstab(IS_P)\}\geq 3$. Finally, by applying again Proposition~\ref{newread} we obtain that $\astab(I)=3$. As a byproduct of computing the $\astab(I)$ we obtain that $\Ass^{\infty}(I)=\Ass(I^3)$. Calculations with Singular show that $\Ass(I^3)$ consists of all prime ideals of height $3,5$ and $6$ which belong to $V^*(I)$, altogether $17$. Moreover we see that in this example, $\dstab(I)< \astab(I)$.

\medskip
As a second example we consider the ideal $I=(xyz,ytu,xzv,tuv,xtv)\subset S=K[x,y,z,t,u,v]$. Then $I=\Sect_{F\in \mathcal{F}(\Delta)}P_F$ where $\Delta$ is the simplicial complex with facets $\mathcal{F}=\{\{z,t\}, \{x,t\}, \{x,u\}, \{y,v\}, \{z,u,v\}\}$ and where $P_F$ is the monomial prime ideal whose generators correspond to the vertices of $F$. The simplicial complex $\Delta$ has no special odd cycles in the sense of \cite{HHTZ}. Thus as a consequence of \cite[Theorem 2.2.]{HHTZ} it follows that the vertex cover algebra of $\Delta$ is standard graded which implies that $\Ass(I)=\Ass^\infty(I)$. Thus $\astab(I)=1$. On the other hand one can check with CoCoA that $\depth S/I = \depth S/I^2=3$ and $\depth S/I^3=2$ . Thus $\astab(I)<3\leq \dstab(I)$.

\medskip
As a last topic of this section we want to recall a few facts about the limit depth of an ideal. As we mentioned already, the function $f(k)=\depth R/I^k$ is constant for $k\gg 0$. We call  $\lim_{k\to \infty}\depth R/I^k$ the {\em limit depth of $I$}, see \cite{HH1}. This limit depth can be computed under certain conditions that we are going to describe now.

Recall that the {\em analytic spread} of an ideal $I$ is the Krull dimension of the fiber ring $\mathcal{R}(I)/\mm \mathcal{R}(I)$. It is known by Brodmann \cite{Br1}  that $$\lim_{k\to \infty}\depth R/I^k\leq n-\ell(I).$$ Thus in particular, if the analytic spread of $I$ is equal to $n$, then $\lim_{k\to \infty}\depth R/I^k=0$.
Eisenbud and Huneke \cite[Proposition 3.3]{EH} showed that equality holds in Brodmann's inequality if the associated graded ring $\gr_I(R)$ is Cohen--Macaulay, which by Huneke \cite{Hu} is the case  if $R$ and $\mathcal{R}(I)$ are Cohen--Macaulay.

In the case that $I$ is a monomial ideal generated in a single degree, the analytic spread of $I$ is the rank  of the integer matrix whose rows correspond to the monomial generators of  $I$.

\section{Polymatroidal ideals and the persistence property}

Discrete polymatroids were introduced in \cite{HH2} and represent a natural generalization of matroids. In the following we recall some basic facts about discrete polymatroids (for more details see \cite{HH2},\cite{HH}).

Let  $\varepsilon_1,\ldots,\varepsilon_n$ denote the canonical basis vectors of $\RR^n$. Let $\RR_{+}^n$ denote the set of vectors $u=(u(1),\ldots,u(n))\in\RR^n$ with each $u(i)\geq 0$. If $u=(u(1),\ldots,u(n))$ and $v=(v(1),\ldots,v(n))$ are two vectors belonging to $\RR_{+}^n$, then we write $u\leq v$ if all components $v(i)-u(i)$ of $v-u$ are nonnegative.  Moreover, we write $u<v$ if $u\leq v$ and $u\neq v$. The modulus of $u=(u(1),\ldots,u(n))\in\RR_{+}^n$ is $|u|=u(1)+\cdots+u(n)$.  Also, let $\ZZ_{+}^n=\RR_{+}^n\sect\ZZ^n$.

A {\em discrete polymatroid} on the ground set $[n]$ is a nonempty finite set $\MP\subset\ZZ_{+}^n$ satisfying the following conditions:
\begin{enumerate}
\item[(1)] if $u\in\MP$ and $v\in\ZZ_{+}^n$ with $v\leq u$, then $v\in\MP$;
\item[(2)] if $u=(u(1),\ldots,u(n))\in\MP$ and $v=(v(1),\ldots,v(n))\in\MP$ with $|u|<|v|$, then there is $i\in [n]$ with $u(i)<v(i)$ such that $u+\varepsilon_i\in\MP$.
\end{enumerate}
A {\em base} of $\MP$ is a vector $u\in\MP$ such that $u<v$ for no $v\in\MP$. The set of all bases of $\MP$ is denoted by $B(\MP)$. It follows from $(2)$ that if $u_1$ and $u_2$ are bases of $\MP$, then $|u_1|=|u_2|$. The modulus of any base of $\MP$ is called the {\em rank} of $\MP$ and denoted by $\rank\MP$. For later proofs it is very useful to have the following characterization of discrete polymatroids.

Let $\MP$ be a nonempty finite set of integer vectors in $\RR_{+}^n$ which contains with each $u\in\MP$ all its integral subvectors, that is, vectors $v$ with $v\leq u$, and let $B(\MP)$ be
the set of vectors $u\in\MP$ with $u<v$ for no $v\in\MP$. Then (see \cite[Theorem 12.2.4]{HH}) $\MP$ is a discrete polymatroid with $B(\MP)$ its set of bases if and only if the following are satisfied
\begin{enumerate}
\item[(i)] all $u\in B(\MP)$ have the same modulus;
\item[(ii)] if $u=(u(1),\ldots,u(n))\in B(\MP)$ and $v=(v(1),\ldots,v(n))\in B(\MP)$ with $u(i)>v(i)$, then there is $j\in [n]$ with $u(j)<v(j)$ such that $u-\varepsilon_i+\varepsilon_j\in B(\MP)$.
\end{enumerate}

Let $\MP$ be a discrete polymatroid on $[n]$ with $B(\MP)$ the set of bases. The {\em polymatroidal ideal} $I$ attached to $\MP$ is the monomial ideal of $S=K[x_1,\ldots,x_n]$ whose set of minimal monomial generators is the set  $G(I)=\{x^u\:\; u\in B(\mathcal{P})\}$. Observe that $I$ is generated in degree $\rank \MP$.

Recall from Section 1 that for any monomial ideal $I\subset S$ and any $i\in [n]$ we have
\[
I_{x_i}=I_{\{i\}}S_{x_i},
\]
where $I_{\{i\}}\subset S_{\{i\}}=K[x_1,\ldots,x_{i-1},x_{i+1},\ldots,x_n]$ is the monomial ideal which is obtained from $I$ by applying the substitution $x_i\mapsto 1$.

\begin{Proposition}
\label{polymlocalization}
Let $I\subset S$ be a polymatroidal ideal. Then for every $i\in [n]$ the ideal   $I_{\{i\}}$  is again polymatroidal.
\end{Proposition}
\begin{proof}
Let $\MP$ be the polymatroid of rank $d$ on the ground set $[n]$ defining the polymatroidal ideal $I$.  Then  $I=(x^u\:\; u\in B(\mathcal{P}))$ is a monomial ideal in $S=K[x_1,\ldots,x_n]$ which is generated in degree $d$. It follows that  $I_{\{i\}}=(x^{u'}\:\; u\in B(\mathcal{P}))\subset S_{\{i\}}$, where for all $u\in B(\MP)$ we set  $x^{u'}=x^u/x_i^{u(i)}$.

We first show that $I_{\{i\}}$ is generated in one degree. More precisely, if $$a_i=\max\{u(i):u\in B(\mathcal{P})\},$$  then we show that
\[
G(I_{\{i\}})=\{x^u/x_i^{a_i}: u\in B(\mathcal{P}),\;  u(i)=a_i\}.
\]

Indeed, let $v\in B(\mathcal{P})$. Then  $v(i)\leq a_i$. We show that there exists $w\in B(\mathcal{P})$ with $w(i)=a_i$ and such that $x^{w'}$ divides $x^{v'}$. This will then yield the desired conclusion. To show this we proceed by induction on $a_i-v(i)$. If $a_i-v(i) =0$, then there is nothing to show. Suppose now that $v(i)<a_i$, and let $u\in B(\mathcal{P})$ with $u(i)=a_i$. Applying the symmetric exchange property (see \cite[Theorem 12.4.1]{HH}), there exists an integer  $j\in [n]$ with  $u(j)<v(j)$ and such that  $u-\varepsilon_i+\varepsilon_j\in B(\mathcal{P})$ and $v_1:=v-\varepsilon_j+\varepsilon_i\in B(\mathcal{P})$. Hence we obtain that $x^{v_1'}$ divides $x^{v'}$. Since $a_i-v_1(i)<a_i-v(i)$, our induction hypothesis implies that there exists  $w\in  B(\mathcal{P})$ with $w(i)=a_i$ and such that $x^{w'}$ divides $x^{v_1'}$. It follows  that $x^{w'}$ divides $x^{v'}$ as well, as desired.

It remains to be shown that the set $B':=\{u':x^{u'}\in G(I_{\{i\}})\}$ is the set of bases of a discrete polymatroid $\mathcal{P'}$ of rank $d-a_i$ on the ground set $[n]\setminus\{i\}$. First notice that for all $u'\in B'$ we have $|u'|=d-a_i$. In order to verify the exchange property, let $u',v'\in B'$ with $u'(k)>v'(k)$. Then we have that $k\neq i$. We may apply now the exchange property for $u,v\in B(\MP)$: $u(k)=u'(k)>v'(k)=v(k)$ then there exists $l\in [n]$ such that $u(l)<v(l)$ and such that the vector $t=u-\varepsilon_k+\varepsilon_l\in B(\MP)$. Since $u(i)=v(i)=a_i$, it follows that $l\neq i$ and $t(i)=a_i$. Therefore we obtain that $t'\in B'$, where $t'=u'-\varepsilon_k+\varepsilon_l$, as desired.
\end{proof}

\begin{Corollary}
\label{romanianstyle}
If $I$ is a polymatroidal ideal, then $I(P)$  is a polymatroidal ideal for all $P\in V^\ast(I)$.
\end{Corollary}

\begin{Proposition}
\label{persistence}
Let $I\subset S$ be a polymatroidal ideal. Then $I$ has the persistence property.
\end{Proposition}
\begin{proof}
Let $k\geq 1$ be an integer. According to Lemma~\ref{localass} we have that  $P\in\Ass(I^k)$ if and  only if $\depth S(P)/I^k(P)=0$.  Note that $I^j(P)=I(P)^j$ for all $j\geq 1$. Moreover, we know from Corollary~\ref{romanianstyle} that  $I(P)$ is again a polymatroidal ideal. Since powers of polymatroidal ideals are again polymatroidal, see \cite[Theorem 12.6.3]{HH}, and since by \cite[Theorem 12.6.2]{HH} polymatroidal ideals have linear resolutions, we conclude that all powers of $I(P)$ have a linear resolution. Now we apply Proposition~\ref{linear} and we obtain that $\depth S(P)/I^j(P)=0$ for all $j\geq k$. But this implies that $P\in \Ass(I^j)$ for all $j\geq k$, as desired.
\end{proof}

Our next goal is to describe the stable set of associated prime ideals of a polymatroidal ideal. For that purpose we first recall  the following result of Villarreal \cite[Proposition 3.11]{Vi}. For the convenience of the reader we present here an alternative proof of it.

\begin{Theorem}
\label{normality}
Let $I\subset S$ be a polymatroidal ideal. Then $\mathcal{R}(I)$ is a normal ring.
\end{Theorem}

\begin{proof}
It is a well-known fact that $\mathcal{R}(I)$ is a normal ring if and only if  $I$ is a normal ideal (see \cite[Proposition 2.1.2]{HSV}). By definition, $I$ is normal if all powers of $I$ are integrally closed. Since a product of polymatroidal ideals is again a polymatroidal ideal (see \cite[Theorem 5.3]{CH}), it is enough to prove that polymatroidal ideals are integrally closed. Since  $I$ is in particular a monomial ideal, it follows from \cite[Theorem 1.4.2]{HH} that $I$ is integrally closed if and only if the following condition is satisfied: for every monomial $u\in S$ and every integer $k$ such that $u^k\in I^k$ we have $u\in I$.

Let $u\in S$ be a monomial of degree $t$ and $k$ an integer such that $u^k\in I^k$. Since $I$ is  generated in one degree, say $d$, it follows from $u^k\in I^k$ that $tk\geq dk$, that is, $t\geq d$. Let $I_l$ be the $K$-subspace of $I$ spanned by all monomials of degree $l$. Then $$(I^k)_{tk}=S_{tk-dk}(I^k)_{dk}=(S_{t-d})^k(I_d)^k=(S_{t-d}I_d)^k.$$

Observe that $S_{t-d}I_d=J_t$ where $J=\mm^{t-d}I$ is a polymatroidal ideal generated in degree $t$. Therefore, we obtain that
\[
u^k\in (I^k)_{tk}=(J_t)^k.
\]
Consequently, we have that $u$ belongs to the integral closure of the base ring $K[J]$. Applying now the normality of $K[J]$ (see \cite[Theorem 12.5.1]{HH}) we obtain that $u\in K[J]$. It follows  that $u\in J$. Therefore $u\in I$, as desired.
\end{proof}

\begin{Corollary}
\label{limitdepth}
Let $I\subset S=K[x_1,\ldots,x_n]$ be a polymatroidal ideal. Then $$\lim_{k\to \infty}\depth S/I^k=n-\ell(I).$$
\end{Corollary}

Combining  Corollary~\ref{romanianstyle} with the preceding corollary one obtains the following algorithm to determine $\Ass^\infty(I)$ for any polymatroidal ideal.

\begin{Algorithm}
{\em Let $I$ be a polymatroidal ideal with $G(I)=\{x^{u_1},\ldots, x^{u_m}\}$, and  let $A$ be the $m\times n$ integer matrix with entries $a_{ij}=u_i(j)$.

Let $F$ be a non-empty subset of $[n]$, and  $v_1,\ldots,v_m$ be the row vectors of the submatrix $(a_{ij})_{i\in [m],\; j\in F}$ of $A$. Furthermore,  let $\{v_{i_1},\ldots, v_{i_r}\}$ be the set of minimal elements among the vectors  $v_1,\ldots,v_m$  with respect to the partial order given by componentwise comparison.
Then $P_F\in\Ass^\infty(I)$ if and only if  $\rank  (a_{i_k, j})_{k=1,\ldots,r,\; j\in F} =|F|$.

\medskip
Thus $\Ass^\infty(I)$ can be determined in finitely many steps.}
\end{Algorithm}

\section{Transversal polymatroids}

\bigskip

Let $F$ be a non-empty subset of $[n]$. As before we denote by $P_F$ the monomial prime ideal $(\{x_i:i\in F\})$. A {\em transversal} polymatroidal ideal is an ideal $I$ of the form
\begin{eqnarray}
\label{classic}
I=P_{F_1} P_{F_2}\cdots P_{F_r},
\end{eqnarray}
 where  $F_1,\ldots,F_r$ is a collection  of non-empty subsets of $[n]$ with $r\geq 1$. It follows from the definition that the product of  transversal polymatroidal ideals is again a transversal polymatroidal ideal. By taking powers of the prime ideal factors of  $I$ which appear several times in (\ref{classic}), we get
\begin{eqnarray}
\label{standard}
I=\prod_{j=1}^s P_{G_j}^{a_j} \quad \text{with}\quad a_j\geq 1,
\end{eqnarray}
where $G_j\neq G_k$ for $j\neq k$.

\begin{Lemma}
\label{unique}
Let $I$ be a transversal polymatroidal ideal. Then $I$ has a unique presentation as in {\em (\ref{standard})}.
\end{Lemma}

\begin{proof}
We proceed by induction on  $s$, the number of different prime factors in the presentation of $I$. So let $s=1$ and $I=P_{G_1}^{a_1}$. We identify $G_1$ as the set of all indices $i$ for which $I_{x_i}=S_{x_i}$. The exponent $a_1$ is the degree of the generators of $I$.

Now let $s>1$ and assume that $I$ has a presentation as in (\ref{standard}). We may further assume that $\Union_{j=1}^sG_j=[n]$. Then  for each $i=1,\ldots,n$ the ideal $I_{x_i}$ determines
\[
I_{\{i\}}= \prod_{j=1\atop i\not\in G_j}^s P_{G_j}^{a_j}.
\]
The transversal polymatroidal ideal $I_{\{i\}}$ has less different prime ideal factors than  $I$ since  $\Union_{j=1}^sG_j=[n]$. Thus our induction hypothesis implies that  the presentation of $I_{\{i\}}$ in the form (\ref{standard}) is unique. For each $j$ such that $G_j\neq [n]$  there exists an integer $i\in[n]$ such that $P_{G_j}^{a_j}$ is a factor of $I_{\{i\}}$. Thus we identified all factors $P_{G_j}^{a_j}$ with $G_j\neq [n]$. The factor $P_{[n]}$ appears with the exponent
\[
d-\sum_{j=1\atop G_j\neq {[n]}}^sa_j,
\]
where $d$ is the degree of the generators of $I$.
\end{proof}

In order to characterize the set of associated prime ideals of $I$  we will introduce a graph $G_I$  associated with  $I$ as follows: the set of vertices $\MV(G_I)$ is the set $\{1,\ldots,r\}$ and $\{i,j\}$ is an edge of $G_I$ if and only if $F_i\sect F_j\neq\emptyset$.

\begin{Example}
\label{transversalgraph}
{\em Let $F_1=\{1,2\}$, $F_2=\{1,2,3,4\}$, $F_3=\{3,5\}$, $F_4=\{4,5\}$ and $I=P_{F_1}\cdots P_{F_4}$ be the transversal polymatroidal ideal of $K[x_1,\ldots,x_5]$. Notice that $I=(x_1,x_2)(x_1,x_2,x_3,x_4)(x_3,x_5)(x_4,x_5)$. Then in Figure~\ref{Fig1} we have depicted the graphs $G_I$ and $G_{I^2}$. One can notice that for any transversal polymatroidal ideal $I$ the graph $G_{I^k}$ is just the $k$-th expansion of $G_I$ (see \cite[Definition 4.2]{FHV}).

\begin{figure}[hbt]
\begin{center}
\psset{unit=1cm}
\begin{pspicture}(2.75,1.5)(8,5)
\rput(1,3){$\bullet$}
\rput(2.5,3){$\bullet$}
\rput(3.5,2.25){$\bullet$}
\rput(3.5,3.75){$\bullet$}
\rput(1,2.6){$1$}
\rput(2.5,2.6){$2$}
\rput(3.7,2){$3$}
\rput(3.7,4){$4$}
\rput(2,1.2){$\bf{G_I}$}
\psline[linewidth=0.6pt,linecolor=black](1,3)(2.5,3)
\psline[linewidth=0.6pt,linecolor=green](2.5,3)(3.5,2.25)
\psline[linewidth=0.6pt,linecolor=blue](3.5,2.25)(3.5,3.75)
\psline[linewidth=0.6pt,linecolor=red](2.5,3)(3.5,3.75)
\rput(6,2.25){$\bullet$}
\rput(7.5,2.25){$\bullet$}
\rput(6,3.75){$\bullet$}
\rput(7.5,3.75){$\bullet$}
\rput(9,1.35){$\bullet$}
\rput(10,1.80){$\bullet$}
\rput(9,4.65){$\bullet$}
\rput(10,4.20){$\bullet$}
\rput(6,1.85){$1$}
\rput(7.5,1.85){$2$}
\rput(6,4.15){$1'$}
\rput(7.5,4.15){$2'$}
\rput(9,0.95){$3'$}
\rput(10,1.40){$3$}
\rput(9,5.05){$4'$}
\rput(10,4.60){$4$}
\rput(7,1.2){$\bf{G_{I^2}}$}
\psline[linewidth=0.6pt,linecolor=black](6,2.25)(7.5,2.25)
\psline[linewidth=0.6pt,linecolor=gray](7.5,2.25)(7.5,3.75)
\psline[linewidth=0.6pt,linecolor=gray](6,2.25)(6,3.75)
\psline[linewidth=0.6pt,linecolor=black](6,2.25)(7.5,3.75)
\psline[linewidth=0.6pt,linecolor=black](6,3.75)(7.5,2.25)
\psline[linewidth=0.6pt,linecolor=black](6,3.75)(7.5,3.75)
\psline[linewidth=0.6pt,linecolor=green](7.5,2.25)(9,1.35)
\psline[linewidth=0.6pt,linecolor=gray](9,1.35)(10,1.80)
\psline[linewidth=0.6pt,linecolor=green](7.5,3.75)(10,1.80)
\psline[linewidth=0.6pt,linecolor=green](7.5,3.75)(9,1.35)
\psline[linewidth=0.6pt,linecolor=green](7.5,2.25)(10,1.80)
\psline[linewidth=0.6pt,linecolor=blue](10,4.20)(10,1.80)
\psline[linewidth=0.6pt,linecolor=gray](9,4.65)(10,4.20)
\psline[linewidth=0.6pt,linecolor=blue](9,1.35)(9,4.65)
\psline[linewidth=0.6pt,linecolor=blue](9,1.35)(10,4.20)
\psline[linewidth=0.6pt,linecolor=blue](9,4.65)(10,1.80)
\psline[linewidth=0.6pt,linecolor=red](7.5,3.75)(9,4.65)
\psline[linewidth=0.6pt,linecolor=red](7.5,2.25)(10,4.20)
\psline[linewidth=0.6pt,linecolor=red](7.5,2.25)(9,4.65)
\psline[linewidth=0.6pt,linecolor=red](7.5,3.75)(10,4.20)
\end{pspicture}
\end{center}
\caption{}
\label{Fig1}
\end{figure}}
\end{Example}

Now we are ready to decide whether the maximal ideal is an associated prime of the transversal polymatroidal ideal $I$ from the connectedness of the graph $G_I$. More precisely, we have

\begin{Theorem}
\label{maxgraphconnected}
Let $I=P_{F_1}\cdots P_{F_r}\subset S$  be a transversal polymatroidal ideal. Then $\mm\in\Ass(I)$ if and only if $\Union_{i=1}^r F_i=[n]$ and $G_I$ is connected.
\end{Theorem}

\begin{proof}
Let us first assume that $\mm\in\Ass(I)$.  Then it follows that  $\Union_{i=1}^r F_i=[n]$. Indeed, let $F:=\Union_{i=1}^r F_i\subsetneq [n]$. Then there exists an ideal $J\subset S(P_F)$ such that $I=JS$. Therefore we have $$\depth_{S} S/I=\depth_{S(P_F)} S(P_F)/J + n-|F|>0,$$ a contradiction.

Assume that $G_I$ is disconnected. It follows from the definition of $G_I$ that after an eventual relabeling of the vertices there exists an integer $l$ such that $1\leq l<r$ and
\begin{eqnarray}
\label{3}
(\Union_{i=1}^l F_i)\sect (\Union_{i=l+1}^r F_i)=\emptyset,
\end{eqnarray}
This implies that there exist integers $s,t\in [n]$ such that
\begin{eqnarray}
\label{1}
s\in (\Union_{i=1}^l F_i) \text{ and } t\in (\Union_{i=l+1}^r F_i).
\end{eqnarray}
Since $\mm\in\Ass(I)$ then there exists a monomial $z\in S\setminus I$ such that $\mm=I:(z)$. From this it follows in particular that $x_sz\in I$ and $x_tz\in I$. By using that $I=P_{F_1}\cdots P_{F_r}$ we obtain that
\begin{eqnarray}
\label{2}
x_sz=x_{i_1}\cdots x_{i_r} \quad  \text{ and } \quad x_tz=x_{j_1}\cdots x_{j_r},
\end{eqnarray}
where $i_k,j_k\in F_k$ for all $k$. It follows now from (\ref{1}) and (\ref{2}) that $s\in\{i_1,\ldots,i_l\}$ and  $t\in\{j_{l+1},\ldots,j_r\}$. Hence  $z\in P_{F_{l+1}}\cdots P_{F_r}$ and  $z\in P_{F_1}\cdots P_{F_l}$. Consequently $$z\in P_{F_1}\cdots P_{F_l}\sect P_{F_{l+1}}\cdots P_{F_r}=P_{F_1}\cdots P_{F_l}\cdot P_{F_{l+1}}\cdots P_{F_r}=I,$$ where the first equality is implied by (\ref{3}). This yields that $z\in I$, a contradiction.

Conversely, assume that  $\Union_{i=1}^r F_i=[n]$ and $G_I$ is connected.  We will prove that $\mm\in\Ass(I)$ by explicitly constructing a monomial $z\in S\setminus I$ such that $I:z=\mm$. Since $G_I$ is connected and has $r$ vertices, we may consider a spanning tree $\MT$ for $G_I$, that is, a collection of $r-1$ edges, say $e_1,\ldots,e_{r-1}$ which cover all vertices of $G_I$ (see the trees $\MT_1$, $\MT_2$ and $\MT_3$ from Example~\ref{treemaxgraph}).

For $k=1,\ldots,r-1$ let $e_k=\{i_k,j_k\}$. Then by the definition of $G_I$ we have that $F_{i_k}\sect F_{j_k}\neq\emptyset$. For any such $k$ we  choose an element  $l_k\in F_{i_k}\sect F_{j_k}$. We define now the monomial $z$ as being $$z=x_{l_1}\cdots x_{l_{r-1}}.$$

We claim that for any $i\in [r]$ we have that $z\in\prod_{j\neq i}P_{F_j}$. If our claim is true then we obtain at once that $I:z=\mm$. Indeed, since $\deg(z)=r-1$ we obtain first that $z\not\in I$. It remains to be shown that $\mm z\subset I$. Let $i\in [n]$ be an arbitrary integer. By our assumption, we have that $\Union_{i=1}^r F_i=[n]$, hence there exists an integer $k$ such that $i\in F_k$. By the claim we have that $z\in\prod_{j\neq k} P_{F_j}$. Therefore, we obtain $$x_iz\in P_{F_k}\cdot\prod_{j\neq k} P_{F_j}=I,$$ as desired.

In order to prove our claim let $i\in [r]$ be an integer. We will reformulate our claim in terms of certain numerical functions on trees  and prove it by induction on $r$. Indeed, it follows from the definition of $z$ that $z\in \prod_{j\neq i} P_{F_j}$ if there exists a function $f\: \{e_1,\ldots, e_{r-1}\}\to [r]$ with $\Im f=[r]\setminus \{i\}$ and such that $f(e_k)\in\{i_k,j_k\}$.

 The case $r=2$ is obvious. Since $\MT$ is a tree, there exists a vertex of degree $1$. For simplicity, we may assume that this vertex is $1$, his only adjacent vertex is $2$ and $e_1=\{1,2\}$. The graph $\MT\setminus\{1\}$ is  a tree with the $r-1$ vertices $\{2,\ldots,r\}$ and edges $\{e_2,\ldots,e_{r-1}\}$. It follows from the induction hypothesis that for every $i\in\{2,\ldots,r\}$ there exists a function $f_i:\{e_2,\ldots,e_{r-1}\}\to \{2,\ldots,r\}$ such that $\Im f_i=\{2,\ldots,r\}\setminus\{i\}$. We may extend these functions to
\[
\tilde{f}_i:\{e_1,\ldots,e_{r-1}\}\mapsto \{1,\ldots,r\},
\]
by setting $\tilde{f}_i(e_1)=1$ and $\tilde{f}_i(e_j)=f_i(e_j)$ for all $j\geq 2$ and obtain that $\Im \tilde{f}_i=[r]\setminus\{i\}$ for all $i\geq 2$. Finally consider $\tilde{f}_1:\{e_1,\ldots,e_{r-1}\}\mapsto \{1,\ldots,r\}$ to be the function defined by $\tilde{f}_1(e_1)=2$ and $\tilde{f}_1(e_j)=f_2(e_j)$ for all $j\geq 2$. It follows that $\Im \tilde{f}_1=[r]\setminus\{1\}$ and we are done.
\end{proof}

\begin{Example}
\label{treemaxgraph}
{\em In the case that $\mm\in\Ass(I)$ the proof of Theorem~\ref{maxgraphconnected} allows us  to compute a monomial $z$ such that $I:z=\mm$. However we may  have several possibilities for choosing $z$ since its choice depends on the spanning tree of $G_I$ and on the intersections of the sets $F_i$ corresponding to the adjacent vertices of the tree. Indeed, let us return to the ideal $I$ from the Example~\ref{transversalgraph}. The graph $G_I$ is connected and we have $3$ spanning trees $\MT_1$, $\MT_2$ and $\MT_3$ depicted below.
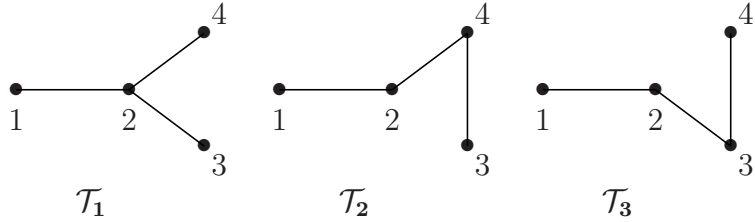
\begin{figure}[hbt]
\begin{center}
\psset{unit=1cm}
\begin{pspicture}(3.75,1.5)(8,4)
\rput(1,3){$\bullet$}
\rput(2.5,3){$\bullet$}
\rput(3.5,2.25){$\bullet$}
\rput(3.5,3.75){$\bullet$}
\rput(1,2.6){$1$}
\rput(2.5,2.6){$2$}
\rput(3.7,2){$3$}
\rput(3.7,4){$4$}
\rput(2,1.5){$\bf{\MT_1}$}
\psline[linewidth=0.6pt](1,3)(2.5,3)
\psline[linewidth=0.6pt](2.5,3)(3.5,2.25)
\psline[linewidth=0.6pt](2.5,3)(3.5,3.75)
\rput(4.5,3){$\bullet$}
\rput(6,3){$\bullet$}
\rput(7,2.25){$\bullet$}
\rput(7,3.75){$\bullet$}
\rput(4.5,2.6){$1$}
\rput(6,2.6){$2$}
\rput(7.2,2){$3$}
\rput(7.2,4){$4$}
\rput(5.5,1.5){$\bf{\MT_2}$}
\psline[linewidth=0.6pt](4.5,3)(6,3)
\psline[linewidth=0.6pt](7,2.25)(7,3.75)
\psline[linewidth=0.6pt](6,3)(7,3.75)
\rput(8,3){$\bullet$}
\rput(9.5,3){$\bullet$}
\rput(10.5,2.25){$\bullet$}
\rput(10.5,3.75){$\bullet$}
\rput(8,2.6){$1$}
\rput(9.5,2.6){$2$}
\rput(10.7,2){$3$}
\rput(10.7,4){$4$}
\rput(9,1.5){$\bf{\MT_3}$}
\psline[linewidth=0.6pt](8,3)(9.5,3)
\psline[linewidth=0.6pt](9.5,3)(10.5,2.25)
\psline[linewidth=0.6pt](10.5,2.25)(10.5,3.75)
\end{pspicture}
\end{center}
\label{Fig2}
\end{figure}

The spanning tree $\MT_1$ gives rise to two such monomials, $x_1x_3x_4$ and $x_2x_3x_4$, since $F_1\sect F_2=\{1,2\}$, $F_2\sect F_3=\{3\}$ and $F_2\sect F_4=\{4\}$. Analogously $\MT_2$ and $\MT_3$ determine the monomials $x_1x_4x_5$, $x_2x_4x_5$ respectively $x_1x_3x_5$, $x_2x_3x_5$. }
\end{Example}

\begin{Corollary}
\label{maxinfinity}
Let $I\subset S$ be a transversal polymatroidal ideal. Then $\mm\in\Ass(I)$ if and only if  $\mm\in\Ass^{\infty}(I).$
\end{Corollary}
\begin{proof}
By Proposition~\ref{persistence}, $I$ satisfies the persistence property.  Therefore  $\mm\in\Ass^{\infty}(I)$ if  $\mm\in\Ass(I)$.  For the converse, let $I=P_{F_1}\cdots P_{F_r}$ and  $\mm\in\Ass^{\infty}(I)$. Then there exists an integer $k\geq 1$ such that $\mm\in\Ass(I^k)$. Since $I^k$ is again a transversal polymatroidal ideal, it follows from Theorem~\ref{maxgraphconnected} that $G_{I^k}$ is connected and $\Union_{i=1}^r F_i=[n]$. One  easily notices that   $G_{I^k}$ is connected  if and only if $G_I$ is connected.  Applying again Theorem~\ref{maxgraphconnected} we get the desired conclusion.
\end{proof}

By using the fact  that the  localization of a transversal polymatroidal ideal is again a transversal polymatroidal ideal we obtain the following

\begin{Corollary}
\label{astab1}
Let $I\subset S$ be a transversal polymatroidal ideal. Then $\astab(I)=1$, that is,  $$\Ass(I)=\Ass^{\infty}(I).$$
\end{Corollary}

\begin{proof}
It follows from Proposition~\ref{persistence} that $\Ass(I)\subset\Ass^{\infty}(I)$. For the converse inclusion, let $P\in\Ass^{\infty}(I)$. Then there exists $k\geq 1$ such that $P\in\Ass(I^k)$. Applying Lemma~\ref{localass} we obtain that $P\in\Ass(S(P)/I^k(P))$. Notice now that $I^k(P)=I(P)^k$ and that $I(P)$ is also a transversal polymatroidal ideal. Since $P$ is the maximal ideal of $S(P)$, it follows from Corollary~\ref{maxinfinity} that $P\in\Ass(S(P)/I(P))$. Therefore, by applying again  Lemma~\ref{localass} we obtain that $P\in\Ass(I)$.
\end{proof}

Next we want to describe the set of associated prime ideals of a transversal polymatroidal ideal $I$. In \cite[Lemma 3.2]{CH} the authors gave a primary decomposition of such a transversal polymatroidal ideal, but unfortunately this primary decomposition is in general far from being irredundant, see also \cite[Proposition 3.4]{CH}. Therefore we cannot read off from their primary decomposition the set of associated prime ideals of a transversal polymatroidal ideal. However, by using the graph $G_I$ this can be done. For this,  to each subgraph $\MH$ of $G_I$ we  associate the prime ideal $P_\MH=\sum_{i\in\MV(\MH)} P_{F_i}$.

\begin{Theorem}
\label{asstransversal}
Let $I\subset S$ be a transversal polymatroidal ideal. Then
\[
\Ass(I)=\{P_\MT\:\;  \MT \text{ is a tree in } G_I\}.
\]
\end{Theorem}

\begin{proof}
Let $I=P_{F_1}\cdots P_{F_r}$.    We prove the statement by induction on $r$. The case $r=1$ is trivial, since in that case $G_I$ is just a vertex. We may assume that $G_I$ is connected. Indeed, let $G_1,\ldots, G_k$ be the connected components of $G_I$, with $k\geq 2$. Then
$
I=I_1\cdots I_k,
$
where
$$
I_j=\prod_{i\in \MV(G_j)}P_{F_i}.
$$
Furthermore, $G_j=G_{I_j}$ for all $j$.  Notice that $I=I_1\cdots I_k=I_1\sect\cdots\sect I_k$, since the ideals $I_j$ are generated in pairwise disjoint  sets of variables. Hence we obtain that $\Ass(I)=\Ass(I_1)\union\ldots\union \Ass(I_k)$, where $I_j$ is a transversal polymatroidal ideal with the associated connected graph $G_{I_j}$.

Obviously we may assume that $\Union_{j=1}^r F_j=[n]$.

 Let $P\in\Ass(I)$.  If $P=\mm$, then by Theorem~\ref{maxgraphconnected} we have that   $P=P_\MT$, where $\MT$ is a spanning tree of $G_I$. Otherwise  there exists an integer $i\in [n]$ such that $x_i\not\in P$. Then $P\in \Ass(I_{\{i\}})$ where
\[
I_{\{i\}}=\prod_{j=1\atop i\not\in F_j}^r P_{F_j}.
\]
The number of prime factors appearing in $I_{\{i\}}$ is less than $r$, since $\Union_{j=1}^r F_j=[n]$. Applying now the induction hypothesis, we obtain that $P=P_\MT$ for some tree $\MT$ in $G_{I_{\{i\}}}$. Since $G_{I_{\{i\}}}$ is a subgraph of $G_I$, we obtain the desired conclusion.

Conversely, let $\MT$ be a tree in $G_I$. If $\MT$ is a spanning tree, then we know from the proof of Theorem~\ref{maxgraphconnected} that $P_\MT=\mm\in\Ass(I)$. Therefore we may assume that $\MT$ is a tree in $G_I$ with $|\MV(\MT)|<r$ and $P_\MT\neq\mm$. This implies that there exists an integer $i\in [n]$ such that $x_i\not\in P_\MT$. Then  $\MT$ remains a tree in $G_{I_{\{i\}}}$ since all vertices of $\MT$ belong to $G_{I_{\{i\}}}$. Moreover,  the number of prime factors appearing in $I_{\{i\}}$ is less than $r$. Therefore, by induction hypothesis we obtain that $P_\MT\in\Ass(I_{\{i\}})$ and consequently $P_\MT\in\Ass(I)$.
\end{proof}

\begin{Example}
\label{assexample}
{\em  Consider again the ideal $I$ given in the Example~\ref{transversalgraph}, that is $$I=(x_1,x_2)(x_1,x_2,x_3,x_4)(x_3,x_5)(x_4,x_5).$$ The trees of $G_I$ have $1,2,3$ or $4$ vertices. To the trees with $1$ vertex, that is the vertices, correspond the associated primes $P_{F_1},\ldots,P_{F_4}$. To the trees with $2$ vertices correspond the associated primes $P_{F_1}+P_{F_2},P_{F_2}+P_{F_3},P_{F_2}+P_{F_4},P_{F_3}+P_{F_4}$. All the trees with $3$ and $4$ vertices generate the same associated prime $\mm$. Consequently we obtain that
\[
\Ass(I)=\{(x_1,x_2),(x_1,x_2,x_3,x_4),(x_3,x_5),(x_4,x_5),(x_3,x_4,x_5),(x_1,x_2,x_3,x_4,x_5)\}.
\]
In particular, we also have that the minimal associated primes correspond to vertices of $G_I$. However, as this example shows, in general not all the vertices give rise to minimal prime ideals.
}
\end{Example}

As a consequence of the above theorem we obtain a description of all possible sets of associated prime ideals of a transversal polymatroidal ideal. More precisely we have

\begin{Corollary}
\label{assprimes}
Let $\MF$ be a subset of $2^{[n]}$ such that $\emptyset\not\in\MF$. Assume that $\MF$ satisfies the following condition:
\begin{eqnarray}
\label{poset}
 A\union B\in\MF\quad \text{for all}\quad  A,B\in\MF \quad \text{with}\quad  A\sect B\neq\emptyset.
\end{eqnarray}
Then there exists a transversal polymatroidal ideal $I$ such that $$\Ass(I)=\{P_A\:\;  A\in\MF\}.$$ Conversely, given any transversal polymatroidal ideal $I$, then  the set $$\{A\:\; A\subset [n] \text{ and } P_A\in\Ass(I)\}$$  satisfies condition {\em (\ref{poset})}.
\end{Corollary}
\begin{proof}
Consider $\MF=\{A_1,\ldots,A_r\}$ to be a set of non-empty subsets of $[n]$ which satisfies condition (\ref{poset}). We define the transversal polymatroidal ideal $I\subset S$ to be $I=\prod_{i=1}^r P_{A_i}$. Then since we consider the vertices of $G_I$ as trees as well, we obtain from Theorem~\ref{asstransversal} that
\[
\Ass(I)=\{P_\MT\:\;  \MT \text{ is a tree in } G_I\}\supset \{P_{A_1},\ldots,P_{A_r}\}.
\]
We prove the converse inclusion by showing that $P_{\MT}\in \{P_{A_1},\ldots,P_{A_r}\}$ for any tree $\MT$ of $G_I$. This will be shown by induction on $k$, the number of vertices of a tree of $G_I$. The case $k=1$ is obvious, since the vertices of $G_I$ correspond to all $P_{A_i}$ with $i=1,\ldots,r$. Assume now that $\MT$ is a tree of $G_I$ with the set of vertices $\MV(\MT)=\{i_1,\ldots,i_k\}$. Since $\MT$ is a tree, there exists a vertex of degree $1$. We may assume that this vertex is $i_1$ and furthermore that $\{i_1,i_2\}$ is an edge of $\MT$. Therefore $A_{i_1}\sect A_{i_2}\neq\emptyset$. For the tree $\MT'=\MT\setminus\{i_1\}$ we apply the induction hypothesis and obtain that $P_{\MT'}=P_A$, where $A=\Union_{j=2}^k A_{i_j}\in\MF$. Hence $A_{i_1}\sect A\neq\emptyset$, and by using the fact that $A_{i_1},A\in\MF$ we obtain via (\ref{poset}) that $A_{i_1}\union A\in\MF$. The conclusion follows at once from the equality $P_\MT=P_B$, where $B=A_{i_1}\union A$.

Conversely, let $I=P_{F_1}\cdots P_{F_r}$ be a transversal polymatroidal ideal. Consider now two subsets $A,B$ of $[n]$ such that $A\sect B\neq\emptyset$ and  $P_A,P_B\in\Ass(I)$. By Theorem~\ref{asstransversal} we know that there exist two trees $\MT,\MT'$ of $G_I$ such that $P_A=P_\MT$ and $P_B=P_{\MT'}$. Therefore, we obtain $A=\Union_{i\in\MV(\MT)}F_i$ and $B=\Union_{i\in\MV(\MT')}F_i$. Thus $A\sect B\neq\emptyset$ implies that there exist two vertices $i\in\MV(\MT)$ and $j\in\MV(\MT')$ such that $F_i\sect F_j\neq\emptyset$. Consequently, the subgraph $\MH$ of $G_I$, whose set of vertices $\MV(\MH)$ is $\MV(\MT)\union\MV(\MT')$ and the edges of $\MH$ are the edges of $\MT$ ant $\MT'$, is connected. A spanning tree $\MT''$ of $\MH$ is a tree of $G_I$ and has the property that
\[
P_{\MT''}=\sum_{i\in\MV(\MT'')} P_{F_i}=\sum_{i\in\MV(\MH)} P_{F_i}=\sum_{i\in\MV(\MT)\union\MV(\MT')} P_{F_i}=P_{A\union B}.
\]
Therefore, by applying again Theorem~\ref{asstransversal} we obtain that $P_{A\union B}\in\Ass(I)$, as desired.
\end{proof}

We obtain also from the Theorem~\ref{asstransversal} an irredundant primary decomposition for any power of a transversal polymatroidal ideal. This improves \cite[Lemma 3.2]{CH}, where the authors could give only a primary decomposition of a transversal polymatroidal ideal, which in general was far from being irredundant. Our proof uses their primary decomposition.

\begin{Corollary}
\label{primarydecomp}
Let $I\subset S$ be a transversal polymatroidal ideal with the set of associated prime ideals $\Ass(I)=\{P_1,\ldots,P_l\}$. Consider $\MT_1,\ldots,\MT_l$  maximal trees of $G_I$ such that $P_j=P_{\MT_j}$ for all $j=1,\ldots,l$. Then
\[
I^k=\Sect_{j=1}^l P_j^{ka_j},
\]
is an irredundant primary decomposition of $I^k$ for any $k\geq 1$, where $a_j=|\MV(\MT_j)|$ for all $j$.
\end{Corollary}
\begin{proof}
First we recall that for the transversal polymatroidal ideal $I=P_{F_1}\cdots P_{F_r}$ we have the following primary decomposition (see \cite[Lemma 3.2]{CH})
\begin{eqnarray}
\label{primdec}
I=\Sect_{A\subset [r]\atop A\neq\emptyset} (\sum_{i\in A}P_{F_i})^{|A|}.
\end{eqnarray}
Since $(\sum_{i\in A}P_{F_i})^{|A|}$ is $\sum_{i\in A}P_{F_i}$-primary and $\Ass(I)=\{P_1,\ldots,P_l\}$ it follows that
\[
I=\Sect_{j=1}^l (\Sect_{A\subset [r]\atop \sum_{i\in A}P_{F_i}=P_j} P_j^{|A|}).
\]
In order to obtain the desired irredundant primary decomposition it remains to be shown that $a_j\geq |A|$ for all $A$ with $\sum_{i\in A}P_{F_i}=P_j$.
To see this, let $\MH$  be the induced subgraph of $G_I$ with vertex set $\MV(\MH)=\{i\:\; i\in\MV(G_I) \text{ and } P_{F_i}\subset P_j\}$. Then we have $A\subset \MV(\MH)$. Since $P_j=\sum_{i\in\MV(\MT_j)} P_{F_i}$ it follows that $\MT_j$ is also a tree of $\MH$.

We show that $\MH$ is connected. Indeed, let $i,s$ be two vertices of $\MH$. Since $P_{F_i}\subset \sum_{i\in\MV(\MT_j)} P_{F_i}$, it follows that there exists an integer ${i_0}$ such that $F_i\sect F_{i_0}\neq\emptyset$.  Therefore, $\{i,i_0\}$ is an edge of $\MH$. Similarly, we see that there exists an integer $s_0$ such that $\{s,s_0\}$ is an edge of $\MH$. Hence  there exists a path from $i$ to $s$ and this yields the desired conclusion.

It follows that $\MV(\MT_j)=\MV(\MH)$. Therefore, $|A|\leq |\MV(\MT_j)|=a_j$.

The irredundant primary decomposition for $I^k$ follows at once, since the maximal trees in $G_{I^k}$ that realize $P_j$ have $ka_j$ vertices.
\end{proof}

\begin{Example}
\label{compasstransversal}
{\em For the ideal $I$ introduced in Example~\ref{transversalgraph} we have computed in Example~\ref{assexample} the set of associated prime ideals
\begin{eqnarray*}
\Ass(I)&=&\{P_1,\ldots,P_6\}\\
&=&\{(x_1,x_2),(x_1,x_2,x_3,x_4),(x_3,x_5),(x_4,x_5),(x_3,x_4,x_5),(x_1,x_2,x_3,x_4,x_5)\}.
\end{eqnarray*}
We noticed there that different trees may determine  the same associated prime ideal. For example, $P_6=\mm$ is determined by all spanning trees of $G_I$, all trees of $G_I$ with $3$ vertices and the following trees with $2$ vertices: $\{2,4\}$, $\{2,3\}$. Therefore $a_6=4$, and  for $\MT_6$ we can choose any spanning tree of $G_I$. For each of the trees $\MT_1,\ldots,\MT_5$  there is only one choice.  These trees are determined by their sets of vertices:
\[
\MV(\MT_1)=\{1\},\quad  \MV(\MT_2)=\{1,2\},\quad \MV(\MT_3)=\{3\},\quad \MV(\MT_4)=\{4\},\quad \MV(\MT_5)=\{3,4\}.
\]
Associated with these trees we have $a_1=1$, $a_2=2$, $a_3=1$, $a_4=1$ and $a_5=2$. It follows now from Corollary~\ref{primarydecomp} that the corresponding irredundant primary decomposition of $I$ is
\[
I=P_1\sect P_2^2\sect P_3\sect P_4\sect P_5^2\sect P_6^4.
\]}
\end{Example}

\medskip
By using the graph $G_I$ of a transversal polymatroidal ideal $I$ we get also a formula for $\depth S/I$. More precisely we have

\begin{Theorem}
\label{depthtransversal}
Let $I=P_{F_1}\cdots P_{F_r}\subset S$ be a transversal polymatroidal ideal. Then
\[
\depth S/I= c(G_I) - 1 + n - |\Union_{i=1}^r F_i|,
\]
where by $c(G_I)$ we denote  the number of connected components of the graph $G_I$.
\end{Theorem}
\begin{proof}
We may assume that $\Union_{i=1}^r F_i=[n]$. Indeed, let $A=\Union_{i=1}^r F_i$. Since $I=JS$ for the polymatroidal ideal  $J\subset S(P_A)$ with $G(J)=G(I)$, we have
\[
\depth S/I=\depth S(P_A)/J + n - |A|.
\]
The graphs $G_I$ and $G_J$ are identical, therefore $c(G_I)=c(G_J)$. Consequently, $\depth S(P_A)/J=c(G_J)-1$ implies the desired formula for $\depth S/I$.

Let $k=c(G_I)$. We prove the statement by induction on $k$. If $k=1$,  then $G_I$ is connected. By applying Theorem~\ref{maxgraphconnected} we obtain $\mm\in\Ass(I)$. Hence $\depth S/I=0$, as desired. Assume now that $k\geq 2$, and let $G_1,\ldots, G_k$ be the connected components of $G_I$. As in the proof of Theorem~\ref{asstransversal} we denote by $I_1,\ldots,I_k$ the transversal polymatroidal ideals for which the associated graphs are the connected components of $G_I$. Hence $I=I_1\cdots I_k=I_1\sect\ldots\sect I_k$. Without loss of generality we may assume that $1\leq l_1\leq \ldots \leq l_k$, where for all $j$ $$l_j=|\Union_{i\in \MV(G_j)} F_i|,$$ and that $\Union_{i\in \MV(G_1)} F_i=\{1,\ldots,l_1\}$. Observe that $l_1+\cdots+l_k=n$.

We have two cases to analyze. First we treat the case $l_1=1$. This implies that the ideal $I_1$ is generated by  $x_1$. If for all $j=1,\ldots,k$ we have $l_j=1$ then $k=n$ and the ideal $I$ is principal. Therefore, $\depth S/I=n-1$, as desired. Otherwise $l_k\geq 2$ and $k\leq n-1$. Consider the short exact sequence
\[
0\To  S/(I:(x_1)) \To S/I \To S/(I+(x_1)) \To 0.
\]
Since $S/(I+(x_1))=S/(x_1)$ it follows that $\depth S/(I+(x_1))=n-1$. We also have that $I:(x_1)=I_2\cdots I_k$. By induction hypothesis $\depth S/(I:(x_1))=k-1-1+n-(n-1)=k-1\leq n-2$. By applying Depth Lemma (see \cite[Proposition 1.2.9]{BH}) we obtain that $\depth S/I=k-1$, as desired.

Consider now the second case, that is $l_1\geq 2$. We use the following short exact sequence
\[
0\To  S/I \To  S/{I_1}\dirsum S/(I_2\sect\cdots\sect I_k) \To S/(I_1+ I_2\sect\cdots\sect I_k)  \To 0.
\]
By induction hypothesis we have that $\depth S/{I_1}=n-l_1$ and $\depth S/(I_2\sect\cdots\sect I_k)=k-2+l_1$, since $I_2\sect\cdots\sect I_k=I_2\cdots I_k$. It follows from $\Union_{i=1}^r F_i=[n]$ that $l_1+\ldots+l_k=n$ and consequently that $kl_1\leq n$. Therefore we have $n-l_1\geq k$ and $k-2+l_1\geq k$. This implies that
\[
\depth (S/{I_1}\dirsum S/(I_2\sect\cdots\sect I_k))=\min\{\depth S/{I_1},\depth S/(I_2\sect\cdots\sect I_k)\}\geq k.
\]
Since the ideals $I_1$ and $I_2\cdots I_k=I_2\sect\cdots\sect I_k$ are generated in disjoint sets of variables we obtain that
\[
S/(I_1+ I_2\cdots I_k)\iso S_1/I_1\otimes_{K} S_2/(I_2\cdots I_k),
\]
where $S_1=K[x_1,\ldots,x_{l_1}]$ and $S_2=K[x_{l_1+1},\ldots,x_n]$. Therefore, by the additivity of depth and using the induction hypothesis we have
\[
\depth S/(I_1+ I_2\cdots I_k)=\depth S_1/I_1 + \depth S_2/(I_2\cdots I_k)=0+(k-2)=k-2.
\]
By applying again Depth Lemma we get that $\depth S/I=k-1$, as desired.
\end{proof}

\begin{Remark}
{\em In \cite[Theorem 12.6.7]{HH} the authors classify all Cohen--Macaulay polymatroidal ideals, which turn out to be the principal ideals, the Veronese ideals and the squarefree Veronese ideals. In the special case of transversal polymatroidal ideal one may derive  as a consequence of Theorem~\ref{depthtransversal} and Theorem~\ref{asstransversal} the above result and  obtains that the Cohen--Macaulay transversal polymatroidal ideals are the principal ideals and the Veronese ideals. Indeed, notice  that  Theorem~\ref{asstransversal} implies  $$\dim S/I=n-\min\{|F_i|\:\; i=1,\ldots,r\}.$$
We denote by $a$ the minimal cardinality of a set $F_i$, where $i=1,\ldots,r$. Hence $\dim S/I=n-a$.  We may assume that $\Union_{i=1}^r F_i=[n]$,  and then it follows from Theorem~\ref{depthtransversal}  that $S/I$  is Cohen--Macaulay if and only if $n-a=k-1$, where $k$ represents the number of connected components of $G_I$. Since $n\geq ka$ it follows that $k-1=n-a\geq ka-a=(k-1)a$. Therefore we obtain that this inequality is valid either if $k=1$ or $a=1$. If $k=1$, then $a=n$ and consequently $|F_i|=n$, for all $i=1,\ldots,r$. This implies that $F_i=[n]$ for all $i$ and hence $I=\mm^r$, the Veronese ideal. Otherwise $a=1$ and then $k=n$. In this  case  $G_I$ has $n$ connected components. Hence we obtain that $r=n$ and $|F_i|=1$ for all $i=1,\ldots,n$. This yields that $I$ is a principal ideal.
}
\end{Remark}
As a consequence of Theorem~\ref{depthtransversal} we obtain that $\depth S/I^k$, as a function of $k$, is constant for any transversal polymatroidal ideal and hence we may also compute the analytic spread of $I$.

\begin{Corollary}
\label{limitdepthtran}
Let $I\subset S$ be a transversal polymatroidal ideal. Then $\depth S/I=\depth S/I^k$ for all $k\geq 1$. In particular, we have $\depth S/I=\lim_{k\to \infty}\depth S/I^k$ and $\ell(I)= n-\depth S/I$.
\end{Corollary}

\begin{proof}
Since $c(G_I)=c(G_{I^k})$ for any $k\geq 1$, then by applying Theorem~\ref{depthtransversal} we obtain that $\depth S/I=\depth S/I^k$ for all $k\geq 1$. Therefore we also have $$\lim_{k\to \infty}\depth S/I^k=\depth S/I.$$ By Corollary~\ref{limitdepth} we get the desired formula for $\ell(I)$.
\end{proof}

\bigskip

\section{Ideals of Veronese type}

Fix a positive integer $d$ and non-negative integers $a_1,\ldots,a_n$ with $a_1+\cdots +a_n\geq d$. Let $B\subset\ZZ_{+}^n$ be the set of vectors $u\in\ZZ_{+}^n$ with $u(i)\leq a_i$ for all $i=1,\ldots,n$ and with $|u|= d$. Then $B$ represents the set of bases of a discrete polymatroid $\MP$ on the ground set $[n]$, of rank $d$, which is called a {\em discrete polymatroid of Veronese type}. Its polymatroidal ideal $I\subset S$ is called an {\em ideal of Veronese type} and will be denoted by $I_{d;a_1,\ldots,a_n}$. The following result is an immediate consequence of the definition of an ideal of Veronese type and of Proposition~\ref{polymlocalization}.
	
\begin{Lemma}
\label{basicproperties}
Let $I=I_{d;a_1,\ldots,a_n}\subset S$ be an ideal of Veronese type. Then we have
\begin{enumerate}
\item[(a)] $I^k=I_{kd;ka_1,\ldots,ka_n}$ for every integer $k\geq 1$;
\item[(b)] $I_{\{i\}}= I_{d-b_i;a_1,\ldots,0,\ldots,a_n}\subset K[\{x_j\:\; j\neq i\}]$, where $b_i$ is the maximal degree of the variable $x_i$ in a minimal generator of $I$.
\end{enumerate}
\end{Lemma}
	
There are three particular cases of ideals of Veronese type that we will consider first. The first case is when $d=\sum_{i=1}^n a_i$, that is $I=I_{d;a_1,\ldots,a_n}$ is a principal ideal. Then  $\Ass^{\infty}(I)=\Ass(I)=\{(x_{i_1}),\ldots,(x_{i_r})\}$, where $a_{i_1},\ldots,a_{i_r}$ are all the nonzero integers from $a_1,\ldots,a_n$. Moreover, we have $\depth S/I=n-1$, $\lim_{k\to \infty}\depth S/I^k=n-1$, $\ell(I)=1$ and $\dstab(I)=\astab(I)=1$.

The second case is when $d=1$. Then  $I=I_{1;a_1,\ldots,a_n}$ is a monomial prime ideal. Therefore $\Ass^{\infty}(I)=\Ass(I)=\{I\}$. Furthermore, we have $\depth S/I=n-\height I$, $\lim_{k\to \infty}\depth S/I^k=n-\height I$, $\ell(I)=\height I$ and $\dstab(I)=\astab(I)=1$.

The third  case to be considered is  when there exists $i\in [n]$ such that $a_i=0$. Let $A$ be the subset of $[n]$ defined as $A=\{j\:\; a_j\neq 0\}$. Then $A\neq\emptyset$ and $G(I)\subset S(P_A)$, where $I=I_{d;a_1,\ldots,a_n}$. By the convention made before Lemma~\ref{localass} we identify $I$ with $G(I)S(P_A)$. For simplicity of notation we denote by $J\subset S(P_A)$ the ideal of Veronese type $G(I)S(P_A)$. It follows then that
\[
\Ass^{\infty}_S(I)=\Ass^{\infty}_{S(P_A)}(J).
\]
Furthermore, we have $\astab(I)=\astab(J)$ and $\depth S/I=\depth S(P_A)/J + n - |A|$. In addition, since $I^k$ can be identified with $J^k$, we also have that $\lim_{k\to \infty}\depth S/I^k=\lim_{k\to \infty}\depth S(P_A)/J^k + n - |A|$, $\ell(I)=\ell(J)$ and $\dstab(I)=\dstab(J)$.

\medskip
Due to this considerations we may assume  throughout the rest of this section that $I=I_{d;a_1,\ldots,a_n}\subset S$ is a Veronese type ideal satisfying
\begin{eqnarray}
\label{generalveronese}
d<\sum_{i=1}^n a_i \quad \text{and} \quad d>1 \quad \text{and} \quad a_1,\ldots,a_n\geq 1.
\end{eqnarray}
We recall, that for such ideals of Veronese type there is a precise description of the associated prime ideals given in \cite[Proposition 3.1]{Vl}.

\begin{Proposition}
\label{assveronese}
Let $I=I_{d;a_1,\ldots,a_n}\subset S$ be an ideal of Veronese type satisfying {\em(\ref{generalveronese})} and $A$ a subset of $[n]$. Then
\[
P_A\in\Ass(I) \Longleftrightarrow \sum_{i=1}^n a_i\geq d-1+|A| \quad \text{ and } \quad \sum_{i\not\in A} a_i\leq d-1.
\]
\end{Proposition}

By using this result we prove the following

\begin{Proposition}
\label{stableassveronese}
Let $I=I_{d;a_1,\ldots,a_n}\subset S$ be an ideal of Veronese type satisfying {\em(\ref{generalveronese})}. Then $\Ass^{\infty}(I)=V^*(I)$.
\end{Proposition}	
\begin{proof}
It is obvious that $\Ass^{\infty}(I)\subset V^*(I)$. Conversely, let $P_A\in V^*(I)$ for some subset $A$ of $[n]$. Then there exists a minimal prime ideal $P_B\in\Ass(I)$ such that $P_B\subset P_A$. This implies that $B\subset A$ and furthermore, by applying Proposition~\ref{assveronese}, we have
\[
\sum_{i\not\in A} a_i\leq \sum_{i\not\in B} a_i\leq d-1.
\]
Consequently we get that for any integer $l\geq 1$
\[
\sum_{i\not\in A} la_i\leq l(d-1)\leq ld-1.
\]
Since $I$ satisfies (\ref{generalveronese}) we have $\sum_{i=1}^n a_i\geq d+1$. Then for $k=|A|-1$ we have that
\[
k(\sum_{i=1}^n a_i-d)\geq |A|-1.
\]
Therefore, we get that $\sum_{i=1}^n ka_i>kd-1+|A|$. Combining this inequality with $\sum_{i\not\in A} ka_i\leq kd-1$ and applying  Proposition~\ref{assveronese}, we obtain that $P_A\in\Ass(I_{kd;ka_1,\ldots,ka_n})$. Therefore, by Lemma~\ref{basicproperties}(a), we have $P_A\in\Ass(I^k)$. Thus we get $P_A\in\Ass^{\infty}(I)$, by the persistence property, as desired.
\end{proof}

It follows immediately from Proposition~\ref{stableassveronese} that $\Ass^{\infty}(I)$ is determined by the minimal prime ideals of $I$. According to Proposition~\ref{assveronese} these minimal prime ideals can be determined as follows: $P_F$ is a minimal prime ideal of $I$ if and only if $F$ is a minimal subset of $[n]$ with respect to inclusion satisfying the following inequalities
\[
\sum_{i\not\in F}a_i + \sum_{i\in F}(a_i-1)\geq d-1 \quad \text{and} \quad  \sum_{i\not\in F}a_i\leq d-1.
\]

\medskip
We can say somewhat more about the set $V^*(I)$ for $I=I_{d;a_1,a_2,\ldots,a_n}$. Without  any loss of generality we may assume that  $a_1\geq a_2\geq \cdots \geq a_n$. We will  use the following facts:
\begin{itemize}
\item[(i)] (\cite[Lemma 2.1]{Vl})  $\sqrt{I}$ is squarefree strongly stable, that is, for all monomials $x_F\in I$ and all integers $1\leq i<j\leq n$ such that $j\in F$ and $i\not\in F$ it follows that $x_{(F\setminus\{j\})\union \{i\}}\in I$. Here $x_F=\prod_{i\in F}x_i$ for $F\subset [n]$.
\item[(ii)] (\cite[Lemma 2.3]{Vl})  Let $J$ be a squarefree strongly stable ideal, then the Alexander dual $J^\vee$ of $J$ is also squarefree strongly stable.
\item[(iii)] $P_F\in V^*(J)$ if and only if $x_F\in J^\vee$ for any monomial ideal $J$.
\end{itemize}

Combining (i), (ii) and (iii) we obtain

\begin{Proposition}
\label{combining}
Let $I=I_{d;a_1,\ldots,a_n}$ be an ideal of Veronese type with $a_1\geq a_2\geq \cdots \geq a_n$. Then for all $P_F\in V^*(I)$ and $1\leq i<j\leq n$ with $j\in F$ and $i\not\in F$ it follows that $P_{(F\setminus\{j\})\union \{i\}}\in V^*(I)$.
\end{Proposition}

It is not the case, as one might expect,  that any set $\MF$ of  incomparable monomial prime ideals can be realized as the set of minimal prime ideals of an ideal of Veronese type.
For example, let $\MF=\{(x_1,x_2),(x_3,x_4)\}$. No matter which order of the variables we choose, the ideal $(x_1,x_2)(x_3,x_4)$ is never squarefree strongly stable with respect to the given order of the variables.

\medskip
We now  characterize  the ideals of Veronese type $I\subset S$ satisfying (\ref{generalveronese}) for which $\astab(I)=1$.
	
\begin{Corollary}
\label{stable1}
Let $I=I_{d;a_1,\ldots,a_n}\subset S$ be an ideal of Veronese type satisfying {\em(\ref{generalveronese})}. Then the following conditions are equivalent:
\begin{enumerate}
\item[(a)] $\mm\in\Ass(I)$;
\item[(b)] $\Ass(I)=\Ass^{\infty}(I)$;
\item[(c)] $\sum_{i=1}^n a_i\geq d-1+n$.
\end{enumerate}
\end{Corollary}	

\begin{proof}
By applying Proposition~\ref{assveronese} we obtain that $(a)\Leftrightarrow(c)$, since $\mm=P_{[n]}$. The implication $(b)\implies(a)$ follows from Proposition~\ref{stableassveronese}. For $(a)\implies(b)$, let $P_A\in\Ass^{\infty}(I)$ for some subset $A$ of $[n]$. Since $\mm\in\Ass(I)$ we obtain that $\sum_{i=1}^n a_i\geq d-1+n$. Therefore $\sum_{i=1}^n a_i\geq d-1+|A|$. The inequality $\sum_{i\not\in A} a_i\leq d-1$ follows from the proof of Proposition~\ref{stableassveronese}. Hence, by applying again Proposition~\ref{assveronese}, we obtain that $P_A\in\Ass(I)$, as desired.
\end{proof}

In the following we give an upper bound for the index of stability of any prime $P\in\Ass^{\infty}(I)$ which we define to be  the smallest integer $k$ such that $P\in\Ass(I^k)$.

\begin{Corollary}
\label{primeindex}
Let $I=I_{d;a_1,\ldots,a_n}\subset S$ be an ideal of Veronese type satisfying {\em(\ref{generalveronese})} and $A$ a subset of $[n]$ with  $P_A\in\Ass^{\infty}(I)$. Then the index of stability of $P_A$ is equal to $\lceil \frac{|A|-1}{\sum_{i=1}^n a_i-d}\rceil$. In particular, $$\astab(I)=\lceil\frac{n-1}{\sum_{i=1}^n a_i-d}\rceil,$$ and $\astab(I)\leq n-1$.
\end{Corollary}	

\begin{proof}
Let $k$ be the smallest integer such that $P_A\in\Ass(I^k)$. Then we have that $P_A\in\Ass(I^k)\setminus\Ass(I^{k-1})$. Therefore, by applying Lemma~\ref{basicproperties} and Proposition~\ref{assveronese} this is equivalent to saying that the following inequalities are fulfilled
\[
k(\sum_{i=1}^n a_i-d)\geq |A|-1 > (k-1)(\sum_{i=1}^n a_i-d) \quad \text{ and } \quad \sum_{i\not\in A} a_i\leq d-1.
\]
The first two inequalities imply the desired formula for the index of stability of $P_A$. For the second equality, it is enough to observe that $\astab(I)$ is equal to the index of stability of $\mm=P_{[n]}$. The last inequality of the statement is obvious.
\end{proof}

The upper bound given in Corollary~\ref{primeindex} is sharp since for every integer $n\geq 2$ the ideal of Veronese type $I=I_{n-1;1,\ldots,1}$ has $\astab(I)=n-1$. Moreover, for an ideal of Veronese type $I$ satisfying (\ref{generalveronese}) we have that  $\astab(I_{d;a_1,\ldots,a_n})=n-1$ if and only if $\sum_{i=1}^n a_i=d+1$. In addition, it follows from the discussion of the third particular case (before Proposition~\ref{assveronese}) that, for a fixed integer $k$ with $1\leq k\leq n-1$ the ideal of Veronese type $I=I_{d;a_1,\ldots,a_k,0,\ldots,0}\subset S$ with $\sum_{i=1}^k a_i=d+1$ satisfies $\astab(I)=k$. Therefore the index of stability of a Veronese type ideal can be any integer between  $1$ and $n-1$.

\medskip
It follows from \cite[Theorem 3.3]{HH1} and Lemma~\ref{basicproperties} that for an ideal $I$ of Veronese type we can compute $\depth S/I$, the limit depth and $\dstab I$. More precisely, we have

\begin{Corollary}
\label{dstabveronese}
Let $I=I_{d;a_1,\ldots,a_n}\subset S$ be an ideal of Veronese type satisfying {\em(\ref{generalveronese})}. Then we have
\begin{enumerate}
\item[(a)] $\depth S/I=\max\{0,d+n-1-\sum_{i=1}^n a_i\}$;
\item[(b)] $\depth S/I^k=\max\{0,kd+n-1-\sum_{i=1}^n ka_i\}$;
\item[(c)] $\dstab(I)=\lceil\frac{n-1}{\sum_{i=1}^n a_i-d}\rceil$.
\end{enumerate}
In particular,  $\astab(I)=\dstab(I)$, $\lim_{k\to \infty}\depth S/I^k=0$ and $\ell(I)=n$.
\end{Corollary}

\begin{proof}
(a) was observed in \cite[Theorem 3.3]{HH1}. Notice that (b) follows at once from (a) and Lemma~\ref{basicproperties}(a). Finally, one can immediately obtain that (b) implies (c). The last equalities are obvious.
\end{proof}

\section{Acknowledgments}

This paper was partially written during the visit of the second and third author at the Universit\"at Duisburg-Essen, Campus Essen. The second author wants to thank The Abdus Salam International Centre for Theoretical Physics (ICTP), Trieste, Italy  for supporting her. The third author was supported by a Romanian grant awarded by UEFISCDI, project number $83/2010$, PNII-RU code TE$\_46/2010$, program Human Resources, ``Algebraic modeling of some combinatorial objects and computational applications''.

\end{document}